\documentclass{siamltex1213}
\usepackage{amsfonts,amsmath,amssymb,graphicx}
\usepackage{epstopdf}
\newcommand{\RR}{{\mathbb{R}}}
\newcommand{\CC}{{\mathbb{C}}}

\newcommand{\uu}{{\mathbf u}}
\newcommand{\ff}{{\mathbf f}}

\newtheorem{example}[theorem]{Example}

\begin{document}
\bibliographystyle{plain}
\pagestyle{myheadings}
\markboth{Davide Palitta and Valeria Simoncini}{matrix-equation strategies for convection-diffusion equations}

\title{matrix-equation-based strategies for convection-diffusion  
equations\thanks{Version of December 5, 2014.
This research is supported in part by the FARB12SIMO grant of the Universit\`a di Bologna.}
} 
\author{Davide Palitta\thanks{Dipartimento di Matematica, Universit\`a di Bologna, %
Piazza di Porta S. Donato, 5, I-40127 Bologna, Italy ({\tt davide.palitta3@unibo.it}).} 
\and Valeria Simoncini\thanks{Dipartimento di Matematica, Universit\`a di Bologna, 
Piazza di Porta S. Donato, 5, I-40127 Bologna, Italy ({\tt valeria.simoncini@unibo.it}).}
}
\maketitle
\begin{abstract}
We are interested in the numerical solution of nonsymmetric linear systems arising from the
discretization of 
convection-diffusion partial differential equations with separable coefficients and
dominant convection.
Preconditioners based on the matrix equation formulation
of the problem are proposed, which naturally approximate the original
discretized problem. For certain types of convection coefficients,
we show that the explicit solution of the matrix equation can effectively replace the linear
system solution. Numerical experiments with data stemming from two and three dimensional
problems are reported, illustrating the potential of the proposed methodology.
\end{abstract}

\begin{keywords}
Convection-diffusion equations, iterative solvers, preconditioning, matrix equations.
\end{keywords}

\begin{AMS}
65F10, 65F30  %, 15A06
\end{AMS}

%%%%%%%%%%%%%%%%%%%%%%%%%%%%%%%%%%%%
\section{Introduction}\label{sec:intro}
We are interested in the numerical solution of the convection-diffusion 
partial differential equation
\begin{equation} \label{eqn:convdiff}
 -\epsilon\Delta u + \mathbf{w}\cdot\nabla u= f, \qquad \mbox{in}\quad \Omega \subset \mathbb{R}^d,
\end{equation}
with $d=2,3$,
where $\mathbf{w}$ is the {convection vector}, while $\epsilon$ is the {viscosity parameter}.
In particular, we consider the dominant convection case, that is $|\mathbf{w}|\gg \epsilon$,
and we assume that $\mathbf{w}$ is incompressible, that is $\mbox{div}(\mathbf{w})=0$. Moreover,
we assume that the components of $\mathbf{w}$ are separable functions in the space variables.
For simplicity the equation is equipped with Dirichlet boundary conditions; 
the analyzed procedures could be used with Neumann boundary conditions as well.

Standard finite difference or finite element discretizations 
%of the differential equation 
yield the algebraic large nonsymmetric linear system
\begin{equation}\label{system}
A\uu =\ff ,
\qquad {\rm with} \quad A\in \mathbb{R}^{N\times N}.
%\qquad {\rm with} \quad A=\epsilon K+N \in \mathbb{R}^{n\times n}.
\end{equation}
%where the symmetric and positive definite matrix $K$ contains the discretization
%coefficients associated with the self-adjoint part of the operator, while
%$N$ contains those associated with the gradient; 
% $A$ also contains the terms associated with the discretized boundary conditions; 
%see, e.g., \cite{ElmanetalbookIIed.2014}. 
The discretization phase is crucial to obtain
a reliable numerical solution to (\ref{eqn:convdiff}), as spurious oscillations may
occur in the approximate solution in the dominant convection regime; see, e.g., \cite{H.C.ELMAN2001}. 
A variety of
strategies is available, which either take a fine enough discretization, or appropriately
modify the differential operator so as to limit
the convection side-effects \cite{ElmanetalbookIIed.2014}; see \cite{Stynes.13} for a recent essay on the subject. 
In our numerical experiments we used the former approach, though
the proposed methodology could also be adapted to handle, e.g., SUPG techniques \cite{Palitta.Ms}.
% a sufficiently fine discretization that ensures a robust numerical solution.

%Upon discretization, problem (\ref{eqn:convdiff}) yields the following 
%large nonsymmetric linear system 
%whose numerical solution has been widely studied in the literature.
We are interested in exploring preconditioning strategies for solving 
(\ref{system}), that stem from knowledge
of the differential operator and its discretization. In particular, a computationally 
cheap approximation to the original differential operator will be employed, that
exploits the original matrix equation structure of the problem.
The idea of using a simplified operator as a preconditioner is well known
in the convection-diffusion equation literature. More precisely, classical strategies
use the symmetric part of the operator -- %has been used as preconditioner --
see, e.g., \cite{Elman.Schultz.86},\cite{Axelsson.Karatson.03}, and in particular
the discussion in \cite{Manteuffel.Otto.93} --
 due to the fact that
a symmetric problem is usually  cheaper to solve than the original nonsymmetric
system if fast solvers such as fast Fourier transform, cyclic reduction or multigrid can be employed.
In the case of nonsymmetric preconditioners, 
attempts have focussed on simplifying the operator so as to still use
fast solvers, and this is achieved, e.g., 
by imposing constant coefficients or by dropping some of the first order terms;
see, e.g., \cite{Axelsson.Karatson.07},\cite{Chinetal.84},\cite{Elman.Golub.90},\cite{Elman.Schultz.86}.
%some of the few attempts have used constant coefficient approximations
%for the first order terms, 
However, their effectiveness and computational
cost of their application have not been clearly
assessed. 
Finally, most publications aim at analyzing the two-dimensional (2D)
problem, whereas the three-dimensional (3D) problem represents a challenging
task, especially from a computational point of view.

We exploit the matrix structure of the discretized problem to construct a cost-effective nonsymmetric
approximation: the application of the new preconditioner amounts to solving
a Sylvester matrix equation, whose coefficient matrices are tightly related to
the original discretized problem. Other authors have used matrix equations either as
a solver or as a preconditioner  for the system stemming from (\ref{eqn:convdiff}); 
see, e.g., \cite{Wachspress.63}, \cite{Wachspress.84b}, \cite{Starke.91b}, \cite{Simoncini1996b}.
However, on the one hand, only very simplified models have been considered, on the other
hand, to the best of our knowledge
no performance evaluation has ever been performed with respect to more
common multilevel techniques.
We derive a new preconditioning operator by first writing down the
full multiterm matrix equation corresponding to finite difference discretization
of the {two or three-dimensional problem and then by appropriately simplifying the operator; then we
iteratively solve}
%recently developed iterative solvers for 
the associated Sylvester equation so as to
achieve good performance of the overall preconditioning phase.
 The choice of the Sylvester equation solver is crucial for assessing the
preconditioning costs. In the literature, ADI and cheap solvers for the Kronecker formulation
 were suggested for various (e.g. symmetric)
variants of the problem \cite{Chinetal.84},\cite{Gunn.64},\cite{Wachspress.84b}; here
%we determine a preconditioner as close as possible to the original operator, while
we exploit a recently developed efficient iterative solver for the Sylvester equation, whose
cost per inner iteration for a 2D problem
may be as low as ${\cal O}(n)$, where $n$ is the one-dimensional problem size; this solver allows us to
preserve the leading coefficients of the two first order derivatives without
losing efficiency.

Our preliminary numerical results are promising and clearly illustrate the potential
of this approach: the overall solution compares rather well with state-of-the-art
and finely tuned algebraic multigrid preconditioners. When used as a solver, in both
2D and 3D problems, its performance is superior to other available approaches.
Finally, our numerical experiments confirm theoretical results in the literature on the
robustness of the approach with respect to the mesh parameter, namely 
the number of iterations of the system solver is essentially independent of the problem size.

Our aim is to derive a proof-of-concept preconditioning procedure for the finite 
difference discretization of the problem on rectangular and parallelepipedal domains,
 as a first step towards its higher level development for the
finite element discretization on more general domains.

Here is a synopsis of the paper. In section~\ref{sec:matrixform} we derive the matrix formulation of
the two-dimensional discretized problem, from which the standard form (\ref{system}) is derived by means of
Kronecker product expansion. Section~\ref{sec:bc} describes how the boundary conditions can be
embedded in the matrix formulation. Section~\ref{sec:precond} uses the previously derived form
to define the preconditioning operator and its associated Sylvester matrix equation, while 
implementation details are given in section~\ref{sec:implementation}. The procedure is then
generalized to three dimensional problems in section~\ref{sec:3D}, while numerical experiments
with systems stemming from
both two and three dimensional data are reported in section~\ref{sec:expes}. Finally, our
conclusions are given in section~\ref{sec:conclusions}.

%Here are some refs \cite{H.C.ELMAN2002}, \cite{H.C.ELMAN2001},
%\cite{J.LIESEN2005},
%\cite{Chinetal.84}

%%%%%%%%%%%%%%%%%%%%%%%%%%%%%%%%%%%%
\section{A matrix oriented formulation}\label{sec:matrixform}
In this section we reformulate the algebraic problem (\ref{eqn:convdiff}) in terms of a
multiterm linear {matrix} equation. This derivation will be used to introduce our
preconditioner, to be applied to a Krylov subspace method as an acceleration strategy
for solving (\ref{system}).
For the ease of presentation, we shall first concentrate on the two-dimensional
problem, and then extend our derivation to the three-dimensional case in section \ref{sec:3D}.

We start by recalling the matrix equation associated with the discretization  by 
five-point stencil finite differences
of the Poisson equation $-\Delta u = f$
on a rectangular domain $\Omega\subset \RR^2$. For the sake of simplicity,
we shall assume that $\Omega=(0,1)^2$. %, however rectangular domains could also be considered.
Let $\Omega_h$ be a uniform discretization of $\Omega$, with nodes  $(x_i,y_j)$,
$i,j=1, \ldots, n-1$. Then assuming homogeneous Dirichlet boundary conditions are
used, centered finite difference discretization leads to the linear system (\ref{system})
with
$$
A = T_{n-1} \otimes I_{n-1} + I_{n-1}\otimes T_{n-1},
$$ 
and $T_{n-1} = {\rm tridiag}(-1, \underline{2},-1)\in\RR^{(n-1)\times (n-1)}$ is the 
symmetric tridiagonal matrix approximating the second-order derivative in one-dimension, 
while the entries of $\uu$ contain an approximation to $u$ at the nodes, having used
a lexicographic order of the entries. %; see, e.g., \cite{addref}.

We thus take a step back, and describe in details the process leading to the Kronecker
formulation, with the aim of deriving its matrix counterpart. This description will
allow us to also include the boundary conditions in a systematic manner.
%, so that more general boundary conditions can be handled.

Let $\bar\Omega_h$ be a uniform discretization of the closed domain $\bar \Omega$,
with equidistant points in each direction, $(x_i,y_j)$, $i,j=0, \ldots, n$.
Analogously, $U_{i,j} = U(x_i, y_j)$ is the value of the
approximation $U$ to $u$ at the nodes.
For each $i,j=1, \ldots, n-1$ we have the usual approximations
$$
u_{xx}(x_i,y_j) \approx \frac{U_{i-1,j} - 2 U_{i,j} + U_{i+1,j}}{h^2} =
\frac 1 {h^2} [1, -2, 1] \begin{bmatrix} U_{i-1,j} \\ U_{i,j} \\ U_{i+1,j} \end{bmatrix} ,
$$
and analogously for the $y$ direction, but from the right,
$$
u_{yy}(x_i,y_j) \approx \frac{U_{i,j-1} - 2 U_{i,j} + U_{i,j+1}}{h^2}
=
\frac 1 {h^2} [U_{i,j-1}, \, U_{i,j},\,  U_{i,j+1} ]  \begin{bmatrix} 1 \\ -2 \\ 1 \end{bmatrix} .
$$
Let
\begin{equation}\label{def_T}
T=-\frac{1}{h^2}\left(\begin{array}{cccccc}
* & * & & & &\\
* & -2 & 1 & & &\\
& 1 & \ddots & \ddots & &\\
&& \ddots & \ddots & 1 &   \\
&& & 1 & -2 & * \\
&&& & * & *\\
\end{array}\right)\in\mathbb{R}^{(n+1)\times (n+1)}; 
\end{equation}
the unspecified values ``*'' are associated with boundary values of $U$ and will be discussed
in section \ref{sec:bc}.
Collecting these relations for all rows $i$'s and for all columns $j$'s, for the whole domain we obtain
$$
-u_{xx} \approx  T U, \qquad
-u_{yy} \approx  U T.
$$
With these approximations we can write  the following classical matrix form of the
finite difference discretization of the Poisson equation on a square domain
(see, e.g., \cite{Wachspress1966})
\begin{eqnarray}\label{eqn:matrixLapl}
TU+UT=F, \qquad {\rm where} \quad F_{i,j} = f(x_i,y_j) + b.c..
\end{eqnarray}
Except for the boundary conditions, 
the Kronecker formulation of (\ref{eqn:matrixLapl}) gives the same form as (\ref{system}).

For the convection-diffusion equation with separable coefficients
 a similar derivation provides a multiterm
linear matrix equation. We state the result in the following proposition, where
separable convection coefficients are assumed. To this end, we define the
matrix 
\begin{eqnarray}\label{eqn:B}
B=\frac{1}{2h}
\begin{bmatrix}
* & * & & &&\\
* & 0 & 1 && &\\
  & -1 & \ddots &\ddots&\\
%& \ddots & \ddots & \ddots &\\
& & \ddots & \ddots &  1 & \\
 & &  & -1  & 0 &* \\
& && & * & *\\
\end{bmatrix}  \in \RR^{(n+1)\times (n+1)},
\end{eqnarray}
which represents the centered finite difference approximation of the first
order one dimensional (1D) derivative
on a uniformly discretized interval.

\begin{proposition}\label{prop:matrixform2D}
Assume that the convection vector $\mathbf{w} = (w_1,w_2)$ satisfies
$w_1=\phi_1(x)\psi_1(y)$ and $w_2=\phi_2(x)\psi_2(y)$. Let
$(x_i, y_j)\in\bar\Omega_h$, $i,j=0, \ldots, n$
and set
$\Phi_k={\rm diag}(\phi_k(x_0), \ldots, \phi_k(x_n))$ and
$\Psi_k={\rm diag}(\psi_k(y_0), \ldots, \psi_k(y_n))$, $k=1,2$.
Then with the previous notation, the centered finite-difference discretization of
the differential operator in (\ref{eqn:convdiff}) leads to the following operator:
% matrix equation to be solved in $U$:
\begin{equation}\label{eqn:matrix1}
{\cal L}_h\, : \, U \to \epsilon TU+\epsilon UT + (\Phi_1B)U\Psi_1+\Phi_2U(B^T\Psi_2) . % = F .
\end{equation}
\end{proposition}

\begin{proof}
The first two terms of ${\cal L}_h(U)$ were derived for (\ref{eqn:matrixLapl}).
%The second order terms correspond to a multiple of the Laplacian, therefore
%their matrix counterpart was already discussed. 
We are left with showing
that the first order term can be expressed in terms of the 1D discretization
matrix $B$ in (\ref{eqn:B}). We have
\begin{eqnarray*} \label{primomatrix}
\phi_1(x_i)\psi_1(y_j) u_x(x_i,y_j) & \approx &
\phi_1(x_i)\frac{u(x_{i+1},y_j)-u(x_{i-1},y_j)}{2h} \psi_1(y_j) \\
&=& \frac 1 {2h}\phi_1(x_i) [ -1, 0, 1] 
\begin{bmatrix} U_{i-1,j} \\ U_{i,j} \\ U_{i+1,j} \end{bmatrix} \psi_1(y_j) ,
\end{eqnarray*}
and analogously,
$$
\phi_2(x_i)\psi_2(y_j)u_y(x_i,y_j)  \approx
\frac 1 {2h}\phi_2(x_i) 
[ U_{i,j-1}, U_{i,j}, U_{i,j+1}] \begin{bmatrix} -1 \\ 0 \\ 1 \end{bmatrix} \psi_2(y_j) .
$$
Collecting these results for all grid nodes and recalling that
$U_{i,j} = U(x_i, y_j)$, we obtain
$$
\left(\phi_1(x_i)\psi_1(y_j) u_{x}(x_i,y_j)\right )_{i,j=0,\ldots, n} \approx 
\Phi_1 B U \Psi_1, 
$$
$$
\left(\phi_2(x_i)\psi_2(y_j) u_{y}(x_i,y_j) \right )_{i,j=0,\ldots, n} \approx 
\Phi_2  U B^T \Psi_2,
$$
and the result follows.
\end{proof}

The operator in (\ref{eqn:matrix1}) is a linear multiterm matrix function, and it is
a general variant of the matrix equations already present in the early 
literature on difference methods \cite{Bickley.McNamee.60}. %,\cite{Starke.91b},\cite{addref}. 
This type of equations now often arises in the discretization
of partial differential equations with stochastic terms; see \cite{Simoncini.survey13} for
a description of this and other contexts where these operators arise.
In the deterministic setting, however, the difficulty of explicitly dealing
with all terms have led the scientific community to abandon this form, in
favor of the Kronecker formulation.

%%%%%%%%%%%%%%%%%%%%%%%%%%%%%%%%%%%%%%%%%%%%%%%%%%%%%%%%
\section{Imposing the boundary conditions}\label{sec:bc}
The algebraic problem needs to be completed by imposing the boundary values. These will
fill up the undefined entries in the coefficient matrices, and in the right-hand side matrix.
To this end, we recall that with the given ordering of the elements in $U$, the first and
last columns, $Ue_1$ and $U e_{n+1}$ resp., 
correspond to the boundary sides $y=0$ and $y=1$, whereas the first and
last rows take up the values at the boundary sides $x=0$ and $x=1$. % sides of the boundary.
With this notation, we wish to complete the corners of the matrices $T$ and $B$, giving rise
to the matrices $T_1, T_2$ and $B_1, B_2$, respectively, so that the following matrix equation
is well defined for $(x_i, y_j) \in \bar \Omega_h$:
\begin{eqnarray}\label{eqn:main}
\epsilon T_1 U + \epsilon U T_2 + \Phi_1B_1 U \Psi_1 + \Phi_2 U B_2\Psi_2 = F .
\end{eqnarray}
With the same notation as for $U$, the entries of $F$ corresponding to $i, j \in \{0,n\}$
will contain contributions from the boundary values of $U$, which are determined next.

For the boundary conditions to be satisfied, the operator ${\cal L}_h (U) =
\epsilon T_1 U + \epsilon U T_2 + \Phi_1B_1 U \Psi_1 + \Phi_2 U B_2\Psi_2$ should act
as the (scaled) identity operator for points at the boundary. To this end, from the generic
matrix $T$ we define the matrix $T_1$ as follows:
$$
T_1=-\frac{1}{h^2}\left(\begin{array}{cccccc}
-1 & 0 & & &&\\
1 & -2 & 1 & &&\\
& 1 & -2 & 1& &\\
& & \ddots & \ddots & \ddots & \\
& & & \ddots& \ddots& 1\\
&&& & 0 & -1\\
\end{array}\right)\in\mathbb{R}^{(n+1)\times (n+1)},
$$
%
%$$
%T_2=T_1^T,
%$$
while the matrix corresponding to the first order operator ($B$ in the generic case) can be
written as
\begin{eqnarray*}
B_1=\frac{1}{2h}\left(\begin{array}{cccccc}
0 & 0 & & &&\\
-1 & 0 & 1 & &&\\
& -1 & 0 &1 & &\\
& & \ddots & \ddots & \ddots & \\
& & & -1& 0 & 1\\
&&& & 0 & 0\\
\end{array}\right)\in\mathbb{R}^{(n+1)\times (n+1)}.
\end{eqnarray*}
%
%and 
%
%$$B_2=B_1^T.$$
%
For the derivative in the $y$ direction, right multiplication should also act like the
identity, therefore we can define $T_2=T_1^T$ and correspondingly, $B_2=B_1^T$.
We are thus ready to define the missing entries in $F$ so that (\ref{eqn:main}) holds. 
For the first column, that is for the side $y=0$, we write
\begin{eqnarray*}
F e_1 & = & (\epsilon T_1 U + \epsilon U T_2 + \Phi_1B_1 U \Psi_1 + \Phi_2 U B_2\Psi_2)e_1 \\
& = & \epsilon T_1 U e_1 + \epsilon U T_2 e_1 + \Phi_1B_1 U \Psi_1 e_1 + \Phi_2 U B_2\Psi_2 e_1 \\
& = & \epsilon T_1 U e_1 + \frac{\epsilon}{h^2} U e_1 + \Psi_1(y_0) \Phi_1B_1 U e_1   ,
\end{eqnarray*}
where we used the fact that $B_2\Psi_2 e_1 = 0$.
 Similar reasonings ensure that the boundary values at $y=1$ are imposed, thus
defining $F e_{n+1}$.
For the side $x=0$ we have
\begin{eqnarray*}
e_1^T F &=&
 e_1^T(\epsilon T_1 U + \epsilon U T_2 + \Phi_1B_1 U \Psi_1 + \Phi_2 U B_2\Psi_2)=e_1^TF \\
&=& e_1^T\epsilon T_1 U + e_1^T \epsilon U T_2 +  e_1^T\Phi_1B_1 U \Psi_1 +  e_1^T\Phi_2 U B_2\Psi_2 \\
&=& \frac{\epsilon}{h^2} e_1^T U + \epsilon e_1^T U T_2 + \phi_2(x_0) e_1^T U B_2\Psi_2 ,
\end{eqnarray*}
where the fact that $e_1^T \Phi_1 B_1 = 0$ was used. The definition of $e_{n+1}^TF$ follows
analogously.

{We stress that maintaining the same boundary conditions in the 
designed preconditioner has other benefits, in addition to accurately reproducing
the original operator. Indeed, it was shown in  \cite{Manteuffel.Parter.90} that
if the preconditioner is itself a matrix arising from the discretization of
an elliptic operator, then under certain regularity assumptions
the $L_2$-norm of the preconditioned matrix $A {\cal P}^{-1}$ is uniformly bounded with
respect to the discretization parameter.}

%%%%%%%%%%%%%%%%%%%%%%%%%%%%%%%%%%%%%%%%%%%%%%%%%%%%%%%%
\section{The new preconditioner}\label{sec:precond}
The main obstacle in directly solving equation (\ref{eqn:matrix1}) is given by
its many general terms. Indeed, had the equation terms with coefficient
matrices on either side of the unknown matrix, the problem
would be readily recognized as a matrix Sylvester equation.
%, which could be solved with state-of-the-art approaches. 
This scenario does occur whenever some of the
convection terms have a simplified structure, so that those terms only depend on some of
the variables. This is the case, for instance, with the operator ${\cal L}(u)=-\Delta u + \phi_1(x) u_x$,
which gives rise to ${\cal L}_h(U) = T_1 U + U T_2 + \Phi_1 B_1 U = (T_1+\Phi_1 B_1) U + U T_2$.
In this case, computational strategies based on the matrix equation
can be employed; see, e.g., \cite{Wachspress.63},\cite{Starke.91b}, and \cite{Ellner1986} and its
references for some early methods; 
see also section~\ref{sec:expes} for some examples with state-of-the-art approaches.
In the general multiterm case, a simplified 
%The aim of this section is to derive a simplified
version of the left-hand side in (\ref{eqn:matrix1}) can be used as an
acceleration operator, in the form of a matrix equation solver; a similar
strategy was adopted in \cite{Chinetal.84}, with encouraging numerical results
with early solution methods.
%consisting only of two appropriately chosen 
%terms, which can be used as an accelerating operator.
%
Let 
$$
{\cal L}_h : Y \mapsto \epsilon T_1Y+\epsilon YT_2 + (\Phi_1B_1)Y\Psi_1+\Phi_2Y(B_2\Psi_2)
$$
be the operator associated with (\ref{eqn:main}).
Then ${\cal L}_h$ can be  approximated by replacing two of the diagonal coefficient matrices, namely
\begin{equation}\label{eqn:means}
\Psi_1\approx {\bar \psi}_1I, \qquad \Phi_2\approx {\bar\phi}_2I,
\end{equation}
where the scalars ${\bar \psi}_1$ and ${\bar\phi}_2$ represent some average of the functions
$\psi_1$ and $\phi_2$ on the given domain. In our experiments we
used 
$$
{\bar \psi}_1:=\frac 1 {n+1} \sum_i \psi_1(x_i), \qquad 
{\bar \phi}_2:=\frac 1 {n+1} \sum_i \phi_2(y_i),
$$
but other strategies may be considered. Substituting the approximations (\ref{eqn:means}) in
${\cal L}_h$, the matrix $Y$ can be collected, yielding the following 
approximation to ${\cal L}_h$:
\begin{equation} \label{Sylv}
{\cal P} : Y \mapsto (\epsilon T+\bar\psi_1\Phi_1B)Y+Y(\epsilon T+\bar\phi_2 B^T \Psi_2) .
\end{equation}
The use of ${\cal P}$ in an acceleration context consists of applying ${\cal P}^{-1}$ to a
given matrix $G$. This corresponds to solving the following Sylvester equation
\begin{eqnarray}\label{eqn:Sylv}
{\cal P}(Y) = G \qquad \Leftrightarrow \qquad
(\epsilon T_1+\bar\psi_1\Phi_1B_1)Y+Y(\epsilon T_2+B_2\Psi_2\bar\phi_2) = G .
\end{eqnarray}
This matrix equation has a unique solution for any $G\ne 0$ if and only if
the spectra of $\epsilon T_1+\bar\psi_1\Phi_1B_1$ and $-(\epsilon T_2+B_2\Psi_2\bar\phi_2)$
have no common eigenvalues. This is ensured for instance
if both coefficient matrices in (\ref{eqn:Sylv}) have eigenvalues in $\CC^+$.
The numerical solution of (\ref{eqn:Sylv}) will be discussed in section \ref{sec:implementation}.

We would like to point out that
the approximation in (\ref{eqn:means}) corresponds to introducing the following
modified convection vector
\begin{equation}\label{vect}
%\tilde{w}=
\widetilde {\mathbf{w}}:=(\bar\psi_1\phi_1(x),\bar\phi_2\psi_2(y))\approx \mathbf{w}.
\end{equation}
so that the continuous operator ${\cal L}(u)=-\epsilon\Delta u +\mathbf{w}\cdot\nabla u$ is approximated by 
$\widetilde {\cal L}(u) = -\epsilon\Delta u + \widetilde {\mathbf{w}}\cdot\nabla u$.
Similar preconditioning strategies relying on simplified (non-self-adjoint)
versions of the operator $\cal L$ have been developed in the past, see, e.g., 
\cite{Axelsson.Karatson.07}, \cite{Chinetal.84},
\cite{Elman.Golub.90}, \cite{Elman.Schultz.86};
some of them use a matrix equation oriented formulation, which allows one to
significantly lower the computational cost, compared with the Kronecker form, while
keeping non-constant coefficients.
In all these cases, however, no performance assessment has been clearly analyzed,
and in fact, the overall cost heavily depends on the  complexity of the method used
to apply the preconditioner. 
 By employing a matrix equation formulation, the computational cost is in general
significantly lower than if one were to solve at each iteration a nonsymmetric
linear system that is very close to the original one (unless special strategies
can be designed for the latter), see, e.g., the discussion 
in \cite[page 1514]{Axelsson.Karatson.07}. 
We refer to \cite{Palitta.Ms} for experimental
comparisons with inner-outer procedures that explore incomplete LU preconditioners.

In the continuous case, our strategy determines a preconditioner that is\footnote{Roughly 
speaking, two elliptic operators are compact -- equivalent if their principal parts
coincide up to a constant factor, and they have homogeneous Dirichlet conditions
on the same part of the boundary.} ``compact -- equivalent'' to the
original operator $\cal L$ \cite[Proposition 3.5]{Axelsson.Karatson.07}. In particular, this
property implies that if one were to use the conjugate gradient method to the
normal equation associated with our preconditioned problem, its convergence rate would
be bounded by a quantity that is independent of the mesh parameter 
\cite[Theorem 4.7]{Axelsson.Karatson.07}.
In our numerical experiments we did not employ the normal equation, but we can still
empirically infer that mesh independence is maintained. 
Finally, we should mention that this property does not necessarily imply that the preconditioner
is good, that is that a low number of iterations is obtained.

%%%%%%%%%%%%%%%%%%%%%%%%%%%%%%%%%%
\subsection{Implementation details}\label{sec:implementation}
Since $A$ in (\ref{system}) is nonsymmetric, the Krylov subspace solver GMRES (\cite{SaadJuly1986})
for nonsymmetric problems can be used, and 
the operator preconditioner ${\cal P}^{-1}$ is applied from the right.
In a standard implementation, at each iteration $k$ of GMRES the preconditioner is
applied to a vector as $y_k = {\cal P}^{-1} g_k$, where $g_k$ is a vector with $N=(n+1)^2$ components.
In our matrix strategy, we first transform the vector $g_k$ into the matrix $G_k$ such
that $g_k = {\rm vec}(G_k)$, and then obtain $y_k = {\rm vec}(Y_k)$ as
the solution to the Sylvester equation
\begin{equation}\label{syilv_it}
(\epsilon T_1+\bar\psi_1\Phi_1B_1)Y+Y(\epsilon T_2+B_2\bar\phi_2\Psi_2) = G_k , 
\end{equation}
where the coefficient matrices have both dimensions $n+1$. 
For $n$ up to a few hundreds, this equation may be efficiently solved by means of
the Bartels-Stewart method \cite{Bartels.Stewart.72}, whose computational cost
is ${\cal O}((n+1)^3)$. This cost should be compared with the cost of solving an
inner system of size $(n+1)^2\times (n+1)^2$ by either a sparse direct solver, or by
a fast solver if applicable.
If a different number of nodes is used in each direction, or a three-dimensional
problem is solved, then the two coefficient
matrices will have different size, without any significant impact on the solution 
process, though the cost will change accordingly.

For larger $n$, the application of the Bartels-Stewart method may become too expensive to
be competitive, therefore an iterative method should be employed. In this case, the
Sylvester equation is solved up to a certain tolerance, so that only an
approximation to $Y$ is obtained. To make the strategy cost and memory
effective when using state-of-the-art iterative methods, the right-hand side should be low rank.
Therefore, $G_k$ is also approximated by a low-rank matrix by means of a truncated
singular value decomposition. The overall effectiveness of the acceleration
process thus depends on how much information the truncation retains, and how
accurate the iterative solution will be. These require the tuning of
three parameters (the truncation and stopping tolerances, and the maximum rank allowed for $G_k$)
that we set a-priori in all our numerical experiments: 
{\tt tol\_truncation}=$10^{-2}$, {\tt tol\_inner}=$10^{-4}$, {\tt r\_max}=$10$. 
The performance
did not seem to be influenced by small variations of the values of these parameters.

Finally, we mention that if an iterative solver is used for (\ref{eqn:Sylv}),
then the preconditioning step makes the overall process a nonlinear operation, as 
the inexact operator ${\cal P}^{-1}$ ``changes'' at each system solver iteration. As a consequence,
the flexible version of GMRES, named FGMRES in the following \cite{Saad1993}, is used, giving rise to
an inner-outer iteration.

The iterative solution of the Sylvester equation is performed by means of the Extended
Krylov subspace method (KPIK), originally proposed in \cite{Simoncini2007b} for the
Lyapunov equation, and then adapted to solving the Sylvester equation in \cite{Breitenetal.2014}.
The relative residual norm is used as stopping criterion, and the residual is checked
at every other iteration, as suggested in \cite{Simoncini2007b}.
Other efficient Sylvester equation solvers could be used, see the recent survey \cite{Simoncini.survey13}.
 However, in a preconditioning framework, the fact that the coefficient matrices can be factorized
once for all makes KPIK very appealing in terms of computational costs; see \cite{Simoncini2007b}
for further details.

%%%%%%%%%%%%%%%%%%%%%%%%%%%%%%%%%%%%
\section{The three-dimensional case}\label{sec:3D}
The 3D convection-diffusion equation can be stated as in (\ref{eqn:convdiff}),  for
$\Omega\subset \RR^3$.
%can be written as 
%\begin{equation}\label{conv3D}
%-\epsilon\Delta+\vec{w}\cdot\nabla u=f \qquad \mbox{in}\quad \Omega \subset \mathbb{R}^3.
% \end{equation}
To convey our idea, we again first focus on the Poisson equation, and then generalize the
matrix formulation to the finite difference discretization of the non-self-adjoint problem (\ref{eqn:convdiff}).
%$$-\Delta u=f, \qquad \mbox{in} \quad \Omega\in\mathbb{R}^3,$$
%and we discretize the problem by five-point stencil finite differences. 
For the sake of simplicity, we shall 
assume that $\Omega=(0,1)^3$, though more general parallelepipedal domains 
could also be considered.
We discretize $\bar\Omega$ with equidistant nodes in each direction, $(x_i,y_j,z_k)$, 
for $i,j, k=0, \ldots, n$. To fix the ideas, 
let \mbox{$U_{i,j}^{(k)} = U(x_i, y_j,z_k)$} denote the value of the
approximation $U$ to $u$ at the node $(x_i, y_j, z_k)$ (other orderings may be
more convenient depending on the equation properties). We also define the tall matrix
$$ 
\mathcal{U}=\left[\begin{array}{c}
U^{(0)} \\
\vdots\\
U^{(n)}\\
\end{array}\right]=\sum_{k=0}^{n}(e_{k+1}\otimes U^{(k)})\in\mathbb{R}^{(n+1)^2\times (n+1)} ,
$$
where %$(U^{(k)})_{i,j}=U(x_i,y_j,z_k) \quad \forall k=1,\dots,n$ and 
$e_{j}$ is the $j$th canonical vector of $\mathbb{R}^{n+1}$.

Let $T$ be as defined in (\ref{def_T}). Then, for $I\in\RR^{(n+1)\times (n+1)}$ the identity matrix,
%$$
%T=-\frac{1}{h^2}\left(\begin{array}{ccccc}
%-2 & 1 & & &\\
%1 & -2 & 1 & &\\
%& \ddots & \ddots & \ddots &\\
%& & \ddots & \ddots & 1 \\
%&& & 1 & -2\\
%\end{array}\right)\in\mathbb{R}^{n\times n},
%$$
%and we have the usual approximation
\begin{eqnarray*}
-u_{xx} &\approx&  \sum_{k=0}^{n}(e_{k+1}\otimes T U^{(k)})=(I\otimes T)\sum_{k=0}^{n}(e_{k+1}\otimes U^{(k)})=
(I\otimes T)\mathcal{U}, \\
-u_{yy} &\approx&  \sum_{k=0}^{n}(e_{k+1}\otimes U^{(k)} T)=\sum_{k=0}^{n}(e_{k+1}\otimes U^{(k)} ) T=
\mathcal{U} T, \\
-u_{zz} &\approx& (T\otimes I)\mathcal{U} .
\end{eqnarray*}
%where $I$ indicates the $n\times n$ identity matrix.

With these approximations we can thus obtain the following matrix form of the
finite difference discretization of the Poisson equation:
\begin{equation}\label{laplacian3D}
(I\otimes T)\mathcal{U}+\mathcal{U} T+(T\otimes I)\mathcal{U}=F,
\end{equation}
where $F=\sum_{k=0}^{n}(e_{k+1}\otimes F^{(k)})\in\mathbb{R}^{(n+1)^2\times (n+1)}$ and 
$(F^{(k)})_{i,j}=f(x_i,y_j,z_k)$.
The Kronecker formulation of the matrix equation (\ref{laplacian3D}) determines  the usual
approximation of the Laplacian operator by seven-point stencil finite differences,
$$
\Delta\approx I\otimes I \otimes T+ I\otimes T\otimes I + T\otimes I\otimes I \in 
\mathbb{R}^{(n+1)^3\times (n+1)^3}.
$$
For the convection-diffusion equation with separable coefficients
 a similar derivation provides a multiterm
linear matrix equation. We state the result in the following proposition.

\begin{proposition}
Assume that the convection vector $\mathbf{w} = (w_1,w_2,w_3)$ satisfies
$w_1=\phi_1(x)\psi_1(y)\upsilon_1(z)$, 
$w_2=\phi_2(x)\psi_2(y)\upsilon_2(z)$, and $w_3=\phi_3(x)\psi_3(y)\upsilon_3(z)$. Let
$(x_i, y_j, z_k)$, $i,j,k=0, \ldots, n$ be the grid nodes discretizing $\bar \Omega$ with mesh size $h$, 
and set
$\Phi_\ell={\rm diag}(\phi_\ell(x_0), \ldots, \phi_\ell(x_n))$, 
$\Psi_\ell={\rm diag}(\psi_\ell(y_0), \ldots, \psi_\ell(y_n))$, and
$\Upsilon_\ell={\rm diag}(\upsilon_\ell(z_0), \ldots, \upsilon_\ell(z_n))$, $\ell=1,2,3$.
%Finally, let
%$$
%B=\frac{1}{2h}\begin{bmatrix}
%0 & 1 & & &\\
%-1 & 0 & 1 & &\\
%& \ddots & \ddots & \ddots &\\
%& & \ddots & \ddots & 1 \\
%&& & -1 & 0\\
%\end{bmatrix} 
%$$
%be the 1D finite-difference discretization of the first order derivative.

Then, with $B$ as defined in (\ref{eqn:B}), the centered finite-difference discretization of
the differential operator in (\ref{eqn:convdiff}) leads to the following operator:
%linear matrix equation to be solved in $\mathcal{U}$:
{\small
\begin{equation}\label{matrix3D}
{\cal L}_h \, : \, U \to (I\otimes T)\mathcal{U}+\mathcal{U}T+(T\otimes I)\mathcal{U}+(\Upsilon_1\otimes\Phi_1B)
\mathcal{U}\Psi_1+(\Upsilon_2\otimes\Phi_2)\mathcal{U} B^T\Psi_2+
[(\Upsilon_3 B)\otimes\Phi_3]\mathcal{U}\Psi_3. %=F.
\end{equation}}
\end{proposition}

\begin{proof}
The second order terms of ${\cal L}_h(U)$ correspond to a multiple of (\ref{laplacian3D}).
We are thus left with showing
that the first order term can be expressed by means of the 1D discretization
matrix $B$. We first fix $k=\bar{k}$ and we have
\begin{eqnarray*} %\label{primomatrix}
\phi_1(x_i)\psi_1(y_j)\upsilon_1(z_{\bar{k}}) u_x(x_i,y_j,z_{\bar{k}}) & \approx &
\upsilon_1(z_{\bar{k}})\phi_1(x_i)\frac{u(x_{i+1},y_j,z_{\bar{k}})-u(x_{i-1},y_j,z_{\bar{k}})}{2h} \psi_1(y_j) \\
&=& \frac 1 {2h}\upsilon_1(z_{\bar{k}})\phi_1(x_i) [ -1, 0, 1] 
\begin{bmatrix} U_{i-1,j}^{(\bar{k})} \\ U_{i,j}^{(\bar{k})} \\ U_{i+1,j}^{(\bar{k})} \end{bmatrix} \psi_1(y_j).
\end{eqnarray*}
Analogously,
$$
\phi_2(x_i)\psi_2(y_j)\upsilon_2(z_{\bar{k}})u_y(x_i,y_j,z_{\bar{k}})  \approx
\frac 1 {2h}\upsilon_2(z_{\bar{k}})\phi_2(x_i) 
[ U_{i,j-1}^{(\bar{k})}, U_{i,j}^{(\bar{k})}, U_{i,j+1}^{(\bar{k})}] \begin{bmatrix} -1 \\ 0 \\ 1 \end{bmatrix} 
\psi_2(y_j) .
$$
Collecting these results for all grid nodes $(x_i,y_j,z_{\bar{k}})_{i,j=0,\dots,n}$ and recalling that
$U_{i,j}^{(k)} = U(x_i, y_j,z_k)$, we obtain
$$
(\phi_1(x_i)\psi_1(y_j)\upsilon_1(z_{\bar{k}}) u_{x}(x_i,y_j,z_{\bar{k}}))_{i,j=0, \ldots, n} \approx 
\upsilon_1(z_{\bar{k}})\Phi_1 B U^{(\bar{k})} \Psi_1,$$
and
$$
(\phi_2(x_i)\psi_2(y_j)\upsilon_2(z_{\bar{k}}) u_{y}(x_i,y_j,z_{\bar{k}}))_{i,j=0, \ldots, n} \approx 
\upsilon_2(z_{\bar{k}})\Phi_2  U^{(\bar{k})} B^T \Psi_2.
$$
Therefore, for all $z$ nodes,
\begin{eqnarray*}
&&  (\phi_1(x_i)\psi_1(y_j)\upsilon_1(z_k) u_{x}(x_i,y_j,z_k))_{i,j,k=0,\ldots, n} \\
&& \hskip 0.7in \approx [\Upsilon_1\otimes I]\sum_{k=0}^n(e_{k+1}\otimes\Phi_1 B U^{(k)} \Psi_1) 
=[\Upsilon_1\otimes I](I\otimes\Phi_1B)[\sum_{k=0}^n(e_{k+1}\otimes  U^{(k)})] \Psi_1 \\
&& \hskip 0.7in =[\Upsilon_1\otimes I](I\otimes\Phi_1B)\mathcal{U} \Psi_1 = (\Upsilon_1\otimes\Phi_1B)\mathcal{U} \Psi_1,
\end{eqnarray*}
and
\begin{eqnarray*}
&&  (\phi_2(x_i)\psi_2(y_j)\upsilon_2(z_z) u_{y}(x_i,y_j,z_k))_{i,j,k=0,\ldots, n} \\
&& \hskip 0.7in  \approx[\Upsilon_2\otimes I]\sum_{k=0}^n(e_{k+1}\otimes\Phi_2  U^{(k)} B^T \Psi_2)
  = [\Upsilon_2\otimes I](I\otimes\Phi_2)[\sum_{k=0}^n(e_{k+1}\otimes U^{(k)})]B^T \Psi_2\\
&& \hskip 0.7in  =[\Upsilon_2\otimes I](I\otimes\Phi_2)\mathcal{U}B^T \Psi_2
  = (\Upsilon_2\otimes\Phi_2)\mathcal{U}B^T \Psi_2.
\end{eqnarray*}
On the other hand, for the $z$ direction it holds
\begin{eqnarray*}
\phi_3(x_i)\psi_3(y_j)\upsilon_3(z_k)u_z(x_i,y_j,z_k)&\approx& \upsilon_3(z_k)\phi_3(x_i)\frac{u(x_i,y_j,z_{k+1})-u(x_i,y_j,z_{k-1})}{2h}\\
&\approx& \frac{1}{2h}\upsilon_3(z_k)\phi_3(x_i) [ -1, 0, 1] 
\begin{bmatrix} U_{i,j}^{(k-1)} \\ U_{i,j}^{(k)} \\ U_{i,j}^{(k+1)} \end{bmatrix} \psi_3(y_j).\\
\end{eqnarray*}
Collecting this relation for all blocks, % we obtain
\begin{eqnarray*}
&&  \left(\phi_3(x_i)\psi_3(y_j)\upsilon_3(z_k) u_{z}(x_i,y_j,z_k)\right)_{i,j,k=0, \ldots, n}  \\
&& \hskip 0.7in \approx(\Upsilon_3B\otimes I)\sum_{k=0}^n[e_{k+1}\otimes (\Phi_3U^{(k)}\Psi_3)]
  =(\Upsilon_3B\otimes I)(I\otimes\Phi_3)[\sum_{k=0}^n(e_{k+1}\otimes U^{(k)})]\Psi_3\\
&& \hskip 0.7in   =(\Upsilon_3B\otimes I)(I\otimes\Phi_3)\mathcal{U}\Psi_3
  =[(\Upsilon_3B)\otimes\Phi_3]\mathcal{U}\Psi_3.
\end{eqnarray*}
and the result follows.
\end{proof}

Imposing the boundary conditions completely determines the entries of $T$ in
all three instances, as well as the missing entries in $B$.
Following the same steps as for the 2D case,
the matrix equation (\ref{matrix3D}) can be written as
{\footnotesize 
$$
( (I\otimes T_1)+ (T_2^T\otimes I)) \,\, \mathcal{U}+
\mathcal{U}\,\, T_3+
(\Upsilon_1\otimes\Phi_1B_1)\,\, \mathcal{U}\,\,\Psi_1+(\Upsilon_2\otimes\Phi_2)\,\,\mathcal{U}\,\, B_3\Psi_2+
[(\Upsilon_3 B_2^T)\otimes\Phi_3]\,\,\mathcal{U}\,\,\Psi_3=F,
$$ 
}
highlighting the presence of five distinct terms in the matrix equation. 
With this ordering of the variables, it holds that $B_3=B_2$ and $T_3=T_2$.
Note that a tensorial
formulation could also be obtained by further ``unrolling'' the Kronecker products in some of
the terms. All three directions of the tensor would have dimension $n+1$, thus being the natural
generalization of the 2D problem. In future research we will explore the possibility 
of explicitly solving the fully tensorized equation (with three terms), possibly using 
recently developed strategies \cite{Hackbusch.Khoro.Tyrty.05},\cite{Kressner.Tobler.10}.

The matrix equation above may take the form of a standard Sylvester equation
depending on whether some of the coefficients $\phi_\ell, \psi_\ell$ and $\upsilon_\ell$, $\ell=1,2,3$ vanish;
see Example~\ref{ex:3D2}.
In the generic case, a preconditioning operator can still be derived, e.g., by averaging the values
of some of the coefficients, the way it was done in the 2D case.
The new preconditioner then consists of an approximation of the equation (\ref{matrix3D}). 
For instance, by approximating $\Psi_1, \Psi_3$ as
$$
\Psi_1\approx \bar \psi_1 I, \qquad \Psi_3\approx\bar\psi_3 I,
$$
and also using $\Phi_2 \otimes \Upsilon_2 \approx \chi I$,
%$$
%\Phi_2 \\approx \beta_1I, \qquad \Upsilon_2\approx\beta_2I.$$
the following operator is obtained,
%With these approximations we are able to write the equation (\ref{matrix3D}) as
%
{\footnotesize
$$
{\cal P} \, : {\cal V} \, \to \, 
\left ( (I\otimes T_1)+ (T_2^T\otimes I) + \bar\psi_1 (\Upsilon_1\otimes\Phi_1B_1) +
\bar\psi_3 [(\Upsilon_3 B_2^T)\otimes\Phi_3]\right )  \mathcal{V}+
\mathcal{V}\, (T_3+ \bar\chi B_3\Psi_2) 
$$
}
whose application entails the solution of a Sylvester equation with coefficient
matrices of size $(n+1)^2\times (n+1)^2$ and $(n+1)\times (n+1)$, respectively.

%%%%%%%%%%%%%%%%%%%%%%%%%%%%%%%%%%%%
\section{Numerical experiments}\label{sec:expes}
In this section several numerical experiments are presented using
both two and three dimensional problems.
Performance with respect to the problem parameter -- the viscosity --
and the discretization parameter -- the meshsize -- are considered.

In both dimensional settings we first consider the case when the matrix equation framework can be
used as a solver for the original equation: this corresponds to problems where
the first order term coefficients only depend on the same variable as the corresponding
derivative. Comparisons with either sparse direct solvers or with iterative solvers
are shown.
Then we report on our experience when
using the matrix equation strategy as a preconditioner for more general convection-diffusion
problems with separable coefficients. We compare the performance of the new approach
with that of state-of-the-art algebraic multigrid preconditioners; experimental comparisons with
(less performing) ILU-type preconditioners can be found in \cite{Palitta.Ms}.
More precisely, we shall consider both
the algebraic multigrid preconditioner MI20 \cite{Boyle2007} with GMRES as a solver, and
AGMG \cite{AGMGuserguide.10} with flexible GMRES as a solver (in AGMG all default parameters
were used, while in MI20, {\tt   control.one\_pass\_coarsen} was set to 1).
Both strategies have been shown to be applicable
to convection-diffusion equations, and in particular, for AGMG it was shown in \cite{Notay.12} that this
variable preconditioning strategy is well suited for both 2D and 3D problems.

We stress that our experimental comparisons somewhat penalize our approach, as the
other preconditioners are in fact fortran90 fully compiled codes, for which a mex file
was made available. Their performance is thus expected to be superior to interpreted
Matlab functions, on which our preconditioner is based. Nonetheless, the reported results
show that the new strategy is still competitive.

Except for Example \ref{ex:ifiss},
all problem data were obtained by centered finite difference discretization on the given
domain; for all experiments the grid fineness was chosen so as to avoid spurious oscillations in the
numerical solution for the coarsest grid used.

All experiments were performed with Matlab Version 7.13.0.564 (R2011b) on a
Dell Latitude laptop running ubuntu 14.04 with 4 CPUs at 2.10GHz.

%%%%%%%%%%%%%%%%%%%%%%%%%%%%%%%%%%%%
\subsection{Two-dimensional problems}\label{sec:expes2D}
Whenever the convection vector has the simplified form
$\mathbf{w}=(w_1,w_2)=(\phi_1(x), \psi_2(y))$, that is $\psi_1$ and $\phi_2$ are constant
functions, the matrix formulation (\ref{eqn:matrix1}) reduces
to 
$$
\epsilon T_1U+\epsilon UT_2 + (\Phi_1B_1)U+U(B_2\Psi_2) = F ,
$$
which is a Sylvester matrix equation, and can thus be solved directly, with no further
approximation. In the first two examples, we thus consider equations leading to this
simplified form, and we compare the performance of either two direct solvers -- in the
matrix and vector equation regimes respectively, or of two iterative solvers.

\begin{example}\label{ex:1}
{\rm
We consider the convection-diffusion (\ref{eqn:convdiff})
 %$$-\epsilon\Delta u +\vec{w}\cdot\nabla u= 0, \qquad \mbox{in}\quad (0,1)\times (0,1),$$
with $\Omega=(0,1)^2$, $f=0$, $\mathbf{w}=(1+(x+1)^2/4,0),$
and the following Dirichlet boundary conditions:
$$\left\{\begin{array}{lc}
         u(x,0)=1 \quad x\in[0,1],\\
         u(x,1)=0\quad x\in[0,1]. \\
        \end{array}\right.
$$
With these data, equation (\ref{eqn:matrix1}) reduces to
$\epsilon T_1U+\epsilon UT_2 + (\Phi_1B_1) U = F$,
that is to the following Sylvester equation,
\begin{equation}\label{eqes1}
(\epsilon T_1+ \Phi_1B_1)U+\epsilon UT_2  = F  ,
\end{equation}
where the right-hand side $F$ contains the Dirichlet boundary conditions (see section~\ref{sec:bc}).
We solve this matrix equation using the {\tt lyap} function in Matlab, from
the Control Toolbox, and its ``vectorized'' version 
(i.e. its Kronecker formulation) by means of a sparse direct solver (Matlab backslash operation).
The results are reported in Table~\ref{tab:1}(left), for a variety of mesh dimensions and
viscosity values; the same number of grid nodes, $n_x$ and $n_y$, was used  
in the $x$ and $y$ directions, respectively.
The numbers show the superiority of the use of the (dense) Sylvester solver, compared with
that of a general-purpose sparse direct solver. 
%Note that in the former approach the code is
%interpreted within Matlab, whereas in the latter case a compiled code is used. This observation
%makes the comparison even more favorable with respect to the matrix equation solver.

The right portion of Table~\ref{tab:1} reports the same type of experiments for the
rectangular domain $\Omega=(0,1)\times (0,2)$. A grid consisting of small squares was still used,
so that a different number of nodes in the two directions was imposed.
Both methods are sensitive to the unbalanced dimension growth in the two directions, although
the matrix oriented approach is still faster. Note that for the largest grid, the performance
of the dense solver {\tt lyap} deteriorates more significantly. For such large problems, an
iterative solver should be preferred.
}
\end{example}

\begin{table}[htb]
\centering
\begin{minipage}{5.5cm}
 {\footnotesize
\begin{tabular}{|r|r|r|r|} 
\hline
$\epsilon$ & $n_x$ &  Vect. &  Matrix \\
&  & {\sc cpu} time  & {\sc cpu} time  \\
\hline
 0.0333  &     65        &        0.02      &        0.008 \\
 0.0333  &    129        &        0.07      &        0.030 \\
 0.0333  &    257        &        0.37      &        0.159 \\
 0.0333  &    513        &        2.01      &        0.898 \\
 0.0333  &  1025         &       10.95      &        6.389 \\
\hline
 0.0167  &     65        &        0.02      &        0.005 \\
 0.0167  &    129        &        0.08      &        0.029 \\
 0.0167  &    257        &        0.39      &        0.155 \\
 0.0167  &    513        &        1.91      &        0.899 \\
 0.0167  &   1025        &       10.95      &        6.354 \\
\hline
 0.0083  &     65        &        0.01      &        0.005 \\
 0.0083  &    129        &        0.08      &        0.032 \\
 0.0083  &    257        &        0.38      &        0.154 \\
 0.0083  &    513        &        1.92      &        0.891 \\
 0.0083  &  1025         &       11.03      &        6.443 \\
\hline
\end{tabular}
}
\end{minipage}
\,\,
\begin{minipage}{6.5cm}
{\footnotesize
\begin{tabular}{|r|r|r|r|r|} 
\hline
$\epsilon$ & $n_x$ & $n_y$ &  Vect. &  Matrix \\
&&& {\sc cpu} time  & {\sc cpu} time  \\
\hline
 0.0333  &     65  &    130      &    0.04       &       0.025 \\
 0.0333  &    129  &    258      &    0.15       &       0.081 \\
 0.0333  &    257  &    514      &    0.76       &       0.415 \\
 0.0333  &    513  &   1026      &    4.28       &       2.410 \\
 0.0333  &   1025  &   2050      &   30.25       &      24.942 \\
\hline
 0.0167  &     65  &    130      &    0.03       &       0.016 \\
 0.0167  &    129  &    258      &    0.15       &       0.093 \\
 0.0167  &    257  &    514      &    0.70       &       0.435 \\
 0.0167  &    513  &   1026      &    5.18       &       2.376 \\
 0.0167  &   1025  &   2050      &   35.31       &      28.691 \\
\hline
 0.0083  &     65  &    130      &    0.03       &       0.018 \\
 0.0083  &    129  &    258      &    0.16       &       0.087 \\
 0.0083  &    257  &    514      &    0.88       &       0.631 \\ 
 0.0083  &    513  &   1026      &    5.15       &       3.245 \\
 0.0083  &   1025  &   2050      &   29.99       &      24.735 \\
\hline 
\end{tabular}
}
\end{minipage}
\caption{Example \ref{ex:1}. Total CPU time for a sparse direct solver and the function {\tt lyap}, as
the viscosity and the problem dimension change. Left: square domain. Right: rectangular domain with
same mesh size and different number of nodes in $x$ and $y$ directions.  \label{tab:1}}
\end{table}

\begin{table}[htb]
\centering
{\footnotesize
\begin{tabular}{|r|r|r|r|r|} 
\hline
$\epsilon$ & $n_x$ & FGMRES+AGMG & GMRES+MI20 & KPIK \\
%\hline
&& {\sc cpu} time   (\# its) & {\sc cpu} time   (\# its) & {\sc cpu} time (\# its)  \\
%\hline
%\hline
\hline
 0.0333  &    129 &    0.1649 (13)    &            0.2843  ( 7)    &          0.2131(24) \\
 0.0333  &    257 &    0.3874 (15)    &            0.5715  ( 8)    &          0.2817(32) \\
 0.0333  &   1025 &   11.4918 (20)    &            8.4540  ( 8)    &          1.5002(44) \\
 0.0333  &   1200 &   12.5441 (17)    &            8.7843  ( 7)    &          1.9722(46) \\
\hline
 0.0167  &    129 &    0.2150 (14)    &            0.2750  ( 7)    &          0.1638(22) \\
 0.0167  &    257 &    0.4356 (14)    &            0.6533  ( 9)    &          0.3628(32) \\
 0.0167  &    513 &    2.0712 (15)    &            2.2171  ( 9)    &          0.6324(38) \\
 0.0167  &   1025 &   11.5428 (18)    &            8.0454  ( 8)    &          2.8454(64) \\
 0.0167  &   1200 &   13.2109 (18)    &            9.5501  ( 8)    &          2.1961(52) \\
\hline
 0.0083  &    129 &    0.1501 (14)    &            0.2685  (10)    &          0.1394(22) \\
 0.0083  &    257 &    0.3651 (15)    &            0.5871  ( 8)    &          0.2885(34) \\
 0.0083  &    513 &    1.6615 (14)    &            2.1814  (10)    &          0.5439(42) \\
 0.0083  &   1025 &   10.0859 (18)    &           10.6729  (11)    &          2.7800(66) \\
 0.0083  &   1200 &   14.4866 (18)    &           11.0856  ( 9)    &          2.8459(58) \\
\hline
\end{tabular}
}
\caption{Example \ref{ex:3}. Total CPU time and number of iterations
for solving the Sylvester equation by iterative methods.
  \label{tab:3}}
\end{table}

\begin{example}\label{ex:3}
{\rm
We consider solving the algebraic problem stemming from the equation
of Example \ref{ex:1} (in $\Omega=(0,1)^2$)
by means of an iterative solver. In the case of the linear system (the Kronecker version of
(\ref{eqes1})), both MI20 and AGMG were considered as preconditioners.
As discussed in section~\ref{sec:implementation}, the iterative solver for 
the Sylvester equation (\ref{eqes1}) was the extended Krylov subspace method
\cite{Simoncini2007b} (KPIK in the following).
In both solvers, the stopping criterion was based on the relative residual 2-norm, with
stopping tolerance equal to $10^{-8}$. Numerical results are reported in Table \ref{tab:3} as the
viscosity and number of grid nodes in each direction change.
We observe that the number of iterations varies for all methods as the mesh is refined, and 
in a more significant manner for KPIK. However, we recall from section \ref{sec:implementation} that
memory requirements in this case are not an issue, as the extended Krylov subspace actually generated is
a subset of $\RR^{n_x}$.  For the two algebraic multigrid
preconditioners, a mildly varying number of iterations with respect to $n_x$ has been largely observed in
the literature, at least experimentally. All methods are not very sensitive to changes in viscosity.
In terms of CPU times, the method based on the matrix equation shows the best performance, especially as
the problem size increases. In summary, iteratively solving the problem by resorting to its matrix equation
formulation pays off in this case. A comparison with the timings of Table \ref{tab:1} is also
of interest. Note that for the largest size, the iterative solvers should be preferred, whereas
for most other sizes, especially in the linear system formulation, the direct solver is superior.
This fact is typical of two-dimensional problems, for which sparse direct solvers remain attractive
over iterative ones, also for very fine discretizations when the Kronecker formulation is used.
}
\end{example}

\vskip 0.05in

In the next two  examples we consider problems where the convection coefficients both depend
on both variables, although in a separable manner. The matrix-equation-based
strategy is then used as a preconditioner, in which the original operator
 is approximated via a Sylvester operator.
%where the convection vector has the form:
%$$\vec{w}=(w_1,w_2)=(\phi_1(x)\psi_1(y),\phi_2(x)\psi_2(y)).$$

\begin{example}\label{ex:ifiss}
{\rm
We consider the convection-diffusion equation in $\Omega=(0,1)^2$ with
 %$$-\epsilon\Delta u +\vec{w}\cdot\nabla u= 0, \qquad \mbox{in}\quad [0,1]\times [0,1],$$
%where
$$
\mathbf{w}=(2y(1-x^2),-2x(1-y^2)),
$$
and the Dirichlet boundary conditions $u(0,y)=0$, $u(1,y)=0$, $u(x,0)=0$ and $u(x,1)=1$. % shown in Figure \ref{B.C.ifiss}.
%
%\begin{figure}[ht]
%\centering
%\includegraphics[scale=0.4]{quadratoBC2-1.eps}
%\caption{Dirichlet Boundary Conditions (Example \ref{ex:ifiss})\label{B.C.ifiss}.}
%\end{figure}
%
%$$u(x)=0, \qquad \mbox{in} \quad [0,1]\times \{ 0\}\cup\{0\}\times [0,1]\cup[0,1]\times\{1\},$$
%$$u(x)=1, \qquad \mbox{in} \quad \{1\}\times [0,1].$$
This is a simple model for the temperature distribution in a cavity with a `hot' external wall ($\{1\}\times [0,1]$) 
and the wind characterized by $\mathbf{w}$ determines a recirculating flow; this model was used as a test
example in \cite{ElmanetalbookIIed.2014}. 
Data were generated with the IFISS package \cite{Elman.Ramage.Silvester.07},
which uses a uniform finite element discretization on a rectangular grid.
The coefficient matrices of the Sylvester operator preconditioner were generated
by finite differences and the corresponding equation solved with KPIK. 
Thanks to the chosen uniform discretization in IFISS, the
obtained preconditioning operator was still effective, in spite of the different discretization
technique. %The Sylvester equation was solved by means of KPIK.
Numerical results are reported in Table \ref{tab:ifiss}, where a stopping tolerance of $10^{-6}$ was
used. The preconditioner MI20 broke down several times on this example, therefore its results
where not included.  Instead, the last column shows timings for the Matlab default sparse
solver (backslash operator).

The numbers in the table show a very good performance of the new approach with respect
to the number of iterations, but timings are in general higher than with AGMG preconditioning;
%Due to the low number of iterations,
 better performance of FGMRES+KPIK may be obtained with a compiled code.
Note that both iterative solvers are largely superior to the sparse direct method.
}
\end{example}

\begin{table}[!ht]
\centering
{\footnotesize
\begin{tabular}{|r|r|r|r|r|} 
\hline
$\epsilon$ & $n_x$ &FGMRES+AGMG &  FGMRES+KPIK & DIRECT\\
&& {\sc cpu} time (\# its) & {\sc cpu} time (\# its) & {\sc cpu} time (\# its) \\
\hline
 0.0050   &   129  &   0.0377 ( 6)        &              0.1301 (10)  &       0.0785 \\
 0.0050   &   257  &   0.1216 ( 6)        &              0.2509 ( 9)  &       0.4338 \\
 0.0050   &   513  &   0.5970 ( 6)        &              0.7709 ( 8)  &       3.3971 \\
 0.0050   &  1025  &   3.1946 ( 7)        &              6.0621 ( 8)  &      19.0872 \\
\hline
 0.0025   &   129  &   0.0634 ( 7)        &              0.1306 (10)  &       0.0824 \\
 0.0025   &   257  &   0.1073 ( 5)        &              0.2173 ( 9)  &       0.4223 \\
 0.0025   &   513  &   0.5029 ( 5)        &              0.7151 ( 8)  &       2.4794 \\
 0.0025   &  1025  &   2.5261 ( 6)        &              5.3601 ( 7)  &      16.1251 \\
\hline
 0.0013   &   129  &   0.1882 (10)        &              0.1686 (10)  &       0.1117 \\
 0.0013   &   257  &   0.2531 ( 7)        &              0.2283 ( 9)  &       0.4317 \\
 0.0013   &   513  &   0.5075 ( 5)        &              0.7088 ( 8)  &       3.2329 \\
 0.0013   &  1025  &   2.1558 ( 5)        &              3.9864 ( 6)  &      24.6895 \\
\hline
\end{tabular}
}
\caption{Example \ref{ex:ifiss}. CPU time and number of iterations for different preconditioners and
a sparse direct solver, as the viscosity and mesh parameter vary.  \label{tab:ifiss}}
\end{table}

\begin{example}\label{ex:elmanramage}
{\rm
We consider  a variant of Example V in \cite{H.C.ELMAN2002}
in $\Omega=(0,1)^2$, where the convection vector
was modified so as to obtain a divergence free convection term, and so as not to have zero mean,
namely
$$
\mathbf{w}=(y(1-(2x+1)^2),-2(2x+1)(1-y^2)),
$$
%$$-\epsilon\Delta u +\vec{w}\cdot\nabla u= 0,$$
together with zero Dirichlet boundary conditions except for the side $y=0$ where
$$\left\{\begin{array}{lc}
         u(x,0)=1 + \tanh[10+20(2x-1)], \quad & 0\leq x \leq 0.5, \\
         u(x,0)=2, \quad & 0.5< x \leq 1.
        \end{array}\right.
$$
The problem was discretized using centered finite differences, and the grid was
selected fine enough so as not to have spurious oscillations.
As in the previous example, we compare the performance of various preconditioners, as
the viscosity and the mesh parameter change. 
 Results are reported in Table~\ref{tab:elmanramage}, where the stopping tolerance $10^{-6}$
was used.

This is a recognized hard problem due to the strong layer appearing near the boundary $y=0$.
This appears to be the only problem where FGMRES+AGMG gives much worse results than the other
methods, both in terms of variability on the number of iterations, and in terms of CPU time.
On the other hand, MI20 and KPIK worked very well as preconditioners, with an
either constant or even decreasing number of iterations, and lower CPU times, with
somewhat lower values for the matrix-equation-based approach.
}
\end{example}

\begin{table}[htb]
\centering
{\footnotesize
\begin{tabular}{|r|r|r|r|r|} 
\hline
$\epsilon$ & $n_x$ &FGMRES+AGMG & GMRES+MI20 &  FGMRES+KPIK\\
&& {\sc cpu} time (\# its) & {\sc cpu} time (\# its) & {\sc cpu} time (\# its) \\
%\hline
\hline
 0.1000   &   128  &   0.1081 (11)      &          0.1989  ( 4)      &        0.5743( 8) \\
 0.1000   &   256  &   0.3335 (14)      &          0.3812  ( 4)      &        0.5351( 7) \\
 0.1000   &   512  &   1.0773 (11)      &          1.1731  ( 4)      &        1.0543( 7) \\
 0.1000   &  1024  &   9.2493 (17)      &          4.3287  ( 4)      &        2.6372( 6) \\
 0.1000   &  2048  &  52.0430 (15)      &         19.6757  ( 4)      &       16.7394( 5) \\
\hline
 0.0500   &   128  &   0.0936 ( 9)      &          0.2269  ( 4)      &        0.5168(10) \\
 0.0500   &   256  &   0.2897 (12)      &          0.3862  ( 4)      &        0.6455( 9) \\
 0.0500   &   512  &   1.2603 (11)      &          1.2380  ( 4)      &        1.1769( 8) \\
 0.0500   &  1024  &  10.3623 (18)      &          4.3345  ( 4)      &        3.0812( 7) \\
 0.0500   &  2048  &  60.5041 (17)      &         20.4056  ( 4)      &       14.9237( 6) \\
\hline
 0.0333   &   128  &   0.0882 ( 8)      &          0.2368  ( 4)      &        0.6428(11) \\
 0.0333   &   256  &   0.2181 ( 9)      &          0.4218  ( 4)      &        0.8149(11) \\
 0.0333   &   512  &   1.5849 (14)      &          1.1977  ( 4)      &        1.3786( 9) \\
 0.0333   &  1024  &   5.4624 (12)      &          4.4130  ( 4)      &        3.7214( 8) \\
 0.0333   &  2048  & 120.9686 (23)      &         20.1120  ( 4)      &       17.9188( 7) \\
\hline
   \end{tabular}
}
\caption{Example \ref{ex:elmanramage}. CPU time and number of iterations for different preconditioners,
as the viscosity and mesh parameter vary.  \label{tab:elmanramage}}
\end{table}

%%%%%%%%%%%%%%%%%%%%%%%%%%%%%%%%%%%%
\subsection{Three-dimensional problems}
Like in the previous section, we first report on the case where the discretized problem can
be directly solved as a Sylvester equation, and then  on the case when the Sylvester equation serves
as preconditioner for the more involved problem.

\begin{example}\label{ex:3D2}
{\rm
We consider the 3D problem $-\epsilon\Delta u+\mathbf{w}\cdot \nabla u=1$,  in $\Omega=(0,1)^3$,
with  convection term
$$\mathbf{w}= (x \sin x , y \cos y, e^{z^2-1}),$$
and zero Dirichlet boundary conditions. % on $\partial \Omega$.
After a centered finite difference discretization with $n_x=n_y=n_z$ nodes in each
direction, the equation (\ref{matrix3D}) takes the form
$$
(I\otimes T_1)\,\,\mathcal{U}+\mathcal{U}\,\,T_3+(T_2^T\otimes I)\,\,\mathcal{U}+(I\otimes\Phi_1B_1)\,\,
\mathcal{U}+\mathcal{U}\,\,B_3\Upsilon_3+[(\Psi_2B_2)^T\otimes I]\,\,\mathcal{U}={\mathbf 1}{\mathbf 1}^T,
$$
where ${\mathbf 1}$ is the vector of all ones,
and this corresponds to the following Sylvester equation in $\mathcal U$,
\begin{equation}\label{silv_ex3D}
[ I\otimes (T_1+\Phi_1B_1)+(T_2+\Psi_2B_2)^T\otimes I ]\,\, \mathcal{U}+\mathcal{U}\,\, (T_3+B_3\Upsilon_3 )={\mathbf 1}{\mathbf 1}^T .
\end{equation}
Note that the fact that the forcing term in the original equation yields a low rank right-hand
side in the matrix equation is crucial for the solution process.

We solved this matrix equation by means of KPIK, and compared its performance with the
solution of the corresponding linear system (via its Kronecker formulation) with
GMRES preconditioned by MI20, and with flexible GMRES preconditioned by AGMG. Results
are reported in Table~\ref{tab:3D2}, with stopping tolerance equal to $10^{-9}$.

As in the corresponding 2D problem (Table \ref{ex:3}), the number of iterations is
quite stable for both algebraic multigrid preconditioners, whereas it varies,
though very mildly, when KPIK is used. On the other hand, CPU time is very much
in favor of the matrix equation solver, whose computation exploited the lower order 
complexity of the problem. One order of magnitude lower timings can be observed
in some instances. All strategies seem not to be sensitive to the
choice of viscosity in the considered range.
}
\end{example}

\begin{table}[ht]
\centering
{\footnotesize
\begin{tabular}{|r|r|r|r|r|r|} 
\hline
$\epsilon$ & $n_x$ &FGMRES+AGMG & GMRES+MI20 & KPIK \\
&& {\sc cpu} time (\# its) & {\sc cpu} time (\# its) & {\sc cpu} time (\# its)  \\  %& {\sc cpu} time (\# its)\\
\hline
 0.0050  &     50  &   0.6044   (13)     &    1.0591  ( 6)  &   0.1353  (18)     \\
 0.0050  &     60  &   1.1250   (14)     &    1.8022  ( 7)  &   0.1734  (18)     \\
 0.0050  &     70  &   2.0385   (14)     &    3.4253  ( 7)  &   0.2326  (20)     \\
 0.0050  &     80  &   3.4803   (14)     &    4.4297  ( 7)  &   0.3583  (20)     \\
 0.0050  &     90  &   5.7324   (15)     &    6.8705  ( 7)  &   0.4999  (22)     \\
 0.0050  &    100  &   8.0207   (15)     &    9.7207  ( 7)  &   0.5677  (22)     \\
\hline
 0.0010  &     50  &   0.6556   (13)     &    1.1306  ( 6)  &   0.1811  (18)     \\
 0.0010  &     60  &   1.3011   (14)     &    1.7854  ( 7)  &   0.2386  (18)     \\
 0.0010  &     70  &   1.9509   (14)     &    2.7829  ( 7)  &   0.2346  (20)     \\
 0.0010  &     80  &   3.5291   (14)     &    4.6576  ( 7)  &   0.4096  (20)     \\
 0.0010  &     90  &   5.1344   (14)     &    6.8176  ( 7)  &   0.4253  (22)     \\
 0.0010  &    100  &   7.6815   (14)     &    9.4935  ( 7)  &   0.5446  (22)     \\
\hline
 0.0005  &     50  &   0.7039   (14)     &    1.0530  ( 6)  &   0.1751  (16)     \\
 0.0005  &     60  &   1.2560   (14)     &    1.7341  ( 6)  &   0.2314  (18)     \\
 0.0005  &     70  &   2.2242   (14)     &    2.9667  ( 7)  &   0.2301  (20)     \\
 0.0005  &     80  &   3.4558   (14)     &    4.5964  ( 7)  &   0.3472  (22)     \\
 0.0005  &     90  &   4.8076   (14)     &    6.4841  ( 7)  &   0.4257  (22)     \\
 0.0005  &    100  &   7.3914   (14)     &    9.6274  ( 7)  &   0.5927  (24)     \\
\hline
   \end{tabular}
}
\caption{Example \ref{ex:3D2}. CPU time and number of iterations for different preconditioners,
as the viscosity and mesh parameter vary.  \label{tab:3D2}}
\end{table}

\begin{example}\label{ex:3D3}
{\rm
Finally, we consider the 3D convection-diffusion equation in $\Omega=(0,1)^3$, with
$f=1$ and convection vector 
$$\mathbf{w}= (yz(1-x^2),0,e^{z}),$$
and zero Dirichlet boundary conditions. % on $\partial \Omega$.
Due to the structure of $\mathbf{w}$, we can treat separately the $(z)$ variable and
the $(x,y)$ variables\footnote{Other variable aggregations are possible. The one we chose
allowed us to explicitly treat the first derivative in the $z$ direction 
in the preconditioner.}. With this variable ordering in mind, we can define
${\cal U}$ of size $n_z \times (n_x n_y)$, and the following corresponding matrix equation
%The equation (\ref{matrix3D}) can thus be written as
%
\begin{equation}\label{ex:3Dprec}
T_3 {\cal U} + {\cal U} ( I\otimes T_2 + T_1^T \otimes I) + \Upsilon_3 B_3 {\cal U}  +
\Upsilon_1 {\cal U} ( I\otimes \Phi_1 B_1) = {\mathbf 1}{\mathbf 1}^T .
%(I\otimes T)\mathcal{U}+\mathcal{U}T+(T\otimes I)\mathcal{U}+(\Sigma_1\otimes \Phi_1B)\mathcal{U}\Psi_1+
%[(\Sigma_3B)\otimes I]\mathcal{U}=1.
\end{equation}
With the approximation $\Upsilon_1\approx \bar\upsilon_1 I$,
 the matrix-equation-based preconditioner is given by  
\begin{equation}\label{syl_3D_prec}
{\cal P} \, : \,{\cal V}\, \to \, 
(T_3+ \Upsilon_3 B_3) {\cal V} + {\cal V} ( I\otimes T_2 + T_1^T \otimes I
+  I\otimes \bar\upsilon_1\Phi_1 B_1)  ,
%[I\otimes T+T\otimes I+\bar\psi_1(\Sigma_1\otimes \Phi_1 B)+(\Sigma_3B)\otimes I]\mathcal{V}+\mathcal{V}T ,
\end{equation}
with coefficient matrices of dimension $n_z\times n_z$ and $n_xn_y \times n_xn_y$.
We compare the performance of flexible GMRES with this preconditioner (the Sylvester equation
is iteratively solved as described in section \ref{sec:implementation}),
with that of flexible GMRES preconditioned by AGMG, and of GMRES with MI20, both applied
to the corresponding linear system. The stopping tolerance was set to $10^{-9}$. 
The results of our numerical experiments are shown in Table~\ref{tab:3D3}.
We readily notice that GMRES+MI20 did not perform well, and we had to stop the
solution process for the two smaller values of the viscosity, due to the
excessive CPU time necessary for building the preconditioner.
The other preconditioning strategies perform comparably all the way up to
$10^6$ unknowns, with a steadily
lower number of iterations for the matrix-equation-based preconditioner, which
in turns implies lower memory requirements of flexible GMRES.
In terms of CPU time performance is comparable, although the new strategy
usually requires lower time; we expect that the use of a compiled implementation
 of this new strategy will make this difference more remarkable.
}
\end{example}

\begin{table}[!ht]
\centering
{\footnotesize
\begin{tabular}{|r|r|r|r|r|} 
\hline
$\epsilon$ & $n_x$ &FGMRES+AGMG  & GMRES+MI20 & FGMRES+KPIK\\
\hline
 0.5000   &    50 &    0.6239   (15)   &      0.9182  ( 6)  &   0.9291  ( 6) \\
 0.5000   &    60 &    1.2095   (15)   &      1.8236  ( 7)  &   1.2027  ( 6) \\
 0.5000   &    70 &    2.1041   (15)   &      3.1649  ( 7)  &   1.6585  ( 6) \\
 0.5000   &    80 &    3.7370   (16)   &      4.9765  ( 7)  &   2.4943  ( 6) \\
 0.5000   &    90 &    7.5874   (16)   &      9.2040  ( 8)  &   3.2513  ( 6) \\
 0.5000   &   100 &    7.7626   (16)   &     11.9912  ( 8)  &   4.7548  ( 6) \\
\hline
 0.1000   &    50 &    0.8041   (17)   &     75.8874  ( 6)  &   1.4610  ( 8) \\
 0.1000   &    60 &    2.1310   (18)   &      - ( -)  &   1.5299  ( 8) \\
 0.1000   &    70 &    2.8043   (18)   &      - ( -)  &   1.8926  ( 8) \\
 0.1000   &    80 &    5.1219   (19)   &      - ( -)  &   3.2928  ( 9) \\
 0.1000   &    90 &    7.3179   (19)   &      - ( -)  &   4.6429  ( 9) \\
 0.1000   &   100 &    9.5759   (19)   &      - ( -)  &   6.5590  ( 9) \\
\hline
 0.0500   &    50 &    0.6780   (17)   &     158.5447  ( 6) &   1.1997  (10) \\
 0.0500   &    60 &    1.4318   (18)   &      - ( -)  &   1.7296  (10) \\
 0.0500   &    70 &    2.8427   (19)   &      - ( -)  &   2.5215  (10) \\
 0.0500   &    80 &    4.9616   (20)   &      - ( -)  &   3.6615  (10) \\
 0.0500   &    90 &    7.1038   (20)   &      - ( -)  &   5.0098  (10) \\
 0.0500   &   100 &   10.7181   (21)   &      - ( -)  &   6.8661  (10) \\
\hline
   \end{tabular}
}
\caption{Example \ref{ex:3D3}. CPU time and number of iterations for different preconditioners,
as the viscosity and mesh parameter vary. ``-'' stands for excessive CPU time in building the
preconditioner. \label{tab:3D3}}
%\caption{Example \ref{ex:3D3}. CPU time and iteration number achieved by applying the different preconditioners with
%the change of the \emph{viscosity parameter} and the dimensions of the problem.\label{tab:3D3}}
\end{table}

%%%%%%%%%%%%%%%%%%%%%%%%%%%%%%%%%%%%%%%%%%%
\section{Conclusions and outlook}\label{sec:conclusions}
We have implemented a new matrix-equation-based strategy for solving or
preconditioning a variety of two and three-dimensional
convection-diffusion problems. Our preliminary numerical experiments
show that the new approach performs comparably well with respect to state-of-the-art
algebraic multigrid preconditioners.
% on a variety of two and three-dimensional convection-diffusion problems. 
As opposed to earlier attempts with this type of approaches, we have shown that 
recently developed Sylvester equation solvers can make the matrix equation
strategy very appealing whenever the original partial differential equation
has separable coefficients.

Our current implementation is limited to the use of uniform meshes for
rectangular or parallelepipedal domains. We plan to generalize the approach to more general
settings in the future. The restriction to separable convection terms can be
relaxed by appropriately approximating the operator at the preconditioning level,
by designing special corresponding separable approximations. These possibilities
will also be explored.

Our description could be generalized to tensor structures with more than three
dimensions, as they occur in many emerging applications 
\cite{Hackbusch.Khoro.Tyrty.05},\cite{Khoromoskij.12},\cite{Kressner.Tobler.10}.  We will 
explore the possibility of extending our
algorithmic methodology to this setting, taking into account the role of the
convection terms in this more advanced framework.

\bibliography{%
/home/valeria/Dropbox/PalittaSimoncini/TESI/bozza_tesi/BIBTEX,%
/home/valeria/Bibl/Biblioteca}

%\bibliography{%
%/home/davide/Scrivania/BIBLIOGRAFIA/BIBTEX.bib}
\end{document}